\documentclass[12pt,reqno]{amsart}

\usepackage{amsfonts,amsmath,amsthm}
\usepackage{amssymb,epsfig}
\usepackage{enumerate} 

\usepackage[text={425pt,650pt},centering]{geometry}
\usepackage{graphicx}
\usepackage{epsfig}
\usepackage{tikz}
\usepackage{caption}
\usepackage{color} 
\definecolor{vert}{rgb}{0,0.6,0}

\numberwithin{figure}{section}

\theoremstyle{plain}
\newtheorem{thm}{Theorem}[section]
\newtheorem{ass}{Main Assumptions}

\newtheorem{defn}{Definition}

\newtheorem{lem}[thm]{Lemma}

\newtheorem{prop}[thm]{Proposition}
\theoremstyle{remark}
\newtheorem{rem}{\bf{Remark}}
\numberwithin{equation}{section}

\newcommand{\N}{\mathbb{N}}

\newcommand{\R}{\mathbb{R}}

\newcommand{\W}{W^{1,\infty}}

\newcommand{\gam}{\gamma}
\newcommand{\del}{\delta}
\newcommand{\ep}{\varepsilon}

\newcommand{\lam}{\lambda}

\newcommand{\Del}{\Delta}

\newcommand{\ol}{\overline}
\newcommand{\ul}{\underline}

\begin{document}

\title
{Well-posedness for constrained Hamilton-Jacobi Equations}

\author
{Yeoneung Kim}

\begin{abstract}

The goal of this paper is to study a Hamilton-Jacobi equation
\begin{equation*}
	\begin{cases}
	u_t=H(Du)+R(x,I(t)) &\text{in }\R^n \times (0,\infty), \\
	\sup_{\R^n} u(\cdot,t)=0 &\text{on }[0,\infty),
	\end{cases}
\end{equation*}
with initial conditions $I(0)=0$, $u_0(x,0)=u_0(x)$ on $\R^n$. Here $(u,I)$ is a pair of unknowns and the Hamiltonian $H$ and the reaction $R$ are given. And $I(t)$ is an unknown constraint (Lagrange multiplier) that forces supremum of $u$ to be always zero. We construct a solution in the viscosity setting using a fixed point argument when the reaction term $R(x,I)$ is strictly decreasing in $I$. We also discuss both uniqueness and nonuniqueness. For uniqueness, a certain structural assumption on $R(x,I)$ is needed. We also provide an example with infinitely many solutions when the reaction term is not strictly decreasing in $I$. 

\end{abstract}

\thanks{Supported in part by NSF grant DMS-1664424}

\address
{
Department of Mathematics, 
University of Wisconsin Madison, 480 Lincoln  Drive, Madison, WI 53706, USA}
\email{yeonkim@math.wisc.edu}

\date{\today}

\maketitle
\tableofcontents

\section{Introduction}
Darwin's theory of evolution suggests that biological individuals evolve under the competition between natural selection and mutation. The mathematical model based on such theory has been studied in literatures (see \cite{Darwin1, Darwin2, Darwin3, Darwin4}). In the model, we usually denote traits, density of population and net birth rate by $x\in \R^n$, $n(x,t)$, $R(x,I)$, respectively, where $I(t)$ represents the total consumptions of the resources of the environment at the time $t$. We can take mutation into account using diffusion $\ep\Del$ for some small $\ep>0$ . The model we consider is the following reaction-diffusion equation
\begin{equation*}
	\begin{cases}
	n^\ep_t - \ep \Delta n^\ep= \frac{n^\ep}{\ep} R(x,I^\ep(t)) &\text{in } \R^n \times (0,\infty),\\
	n^{\ep}(x,0)=n^\ep_0 \in L^1(\R^n) &\text {on }\R^n,
	\end{cases}
\end{equation*}
where we denote the total population associated with the rate $\ep$ by $n^\ep(x,t)$ and $n^\ep_0\geq0$ is a given initial density. Moreover, $I^\ep(t)$ is defined as
\begin{equation*}
I^\ep(t)=\int_{\R^n} \psi(x) n^\ep(t,x)dx,
\end{equation*} 
where $\psi$ is a given smooth, non-negative compactly supported kernel representing consumption rate of resources.

It was studied by G. Barles, S. Mirrahimi, B. Perthame \cite{General.C} that after taking Hopf-Cole transformation $n^\ep(x,t)=e^{u^\ep(x,t) /\ep}$, as mutation rate $\ep$ vanishes, $u^\ep$ converges locally uniformly to $u$ in $\R^n \times [0,\infty)$ which is a solution of the constrained Hamilton-Jacobi equation
\begin{equation*}
	\begin{cases}
	u_t=|Du|^2+R(x,I(t)) &\text{in } \R^n \times (0,\infty),\\
	\max_{ \R^n} u(\cdot,t)=0 &\text{on } [0,\infty),\\
	u(x,0)=u_0(x) &\text{on } \R^n.
	\end{cases}
\end{equation*} 
Motivated from this, we will discuss well-posedness of viscosity solutions for an equation with a general Hamiltonian $H(p)$ and an unknown constraint $I(t)$. The equation we consider is the following
 \begin{equation} \label{original}
     \begin{cases}
       u_t=H(Du)+R(x,I(t)) & \text{in } \R^n \times [0,T],\\
       \sup_{\R^n} u(\cdot ,t)=0 &\text{on } [0,T] ,\\
	 I(0)=0,\\
  	 u(x,0)=u_0(x) & \text{on } \R^n.  
	\end{cases}
\end{equation}

\begin{ass} We need assumptions on $R(x,I): \R^n \times [-2I_M,2I_M]\rightarrow \R$ for $I_M>0$, $u_0(x)$ and $I(t)$, some of which are natural but some are technical.  

	\begin{itemize}
	\item [(A1)] There exist $K_1, K_2>0$ such that $-K_1 \leq R_I(x,I) \leq -K_2$;
	\item [(A2)] $\max_{\R^n} R(\cdot,I_M)=0$;
	\item [(A3)] $\min_{\R^n} R(\cdot,0)=0$;
      \item [(A4)] $\sup_{|I| \leq 2I_M} \|R(\cdot,I)\|_{\W(\R^n)}< \infty$;
      \item [(A5)]  $u_0(x) \in \W(\R^n)$ and $\sup_{x\in \R^n} u_0(x)=0$;
      \item [(H1)]  $H \in C(\R^n, [0,\infty))$ is a nonnegative Hamiltonian with $H(0)=0$ and is locally Lipschitz continuous in $p$.
	\end{itemize} 
\end{ass}
Throughout the paper, the above assumptions are always in force. Additionally,  $f \in \W(\R^n)$, that is; $\|f\|_{L^\infty(\R^n)}+\|Df\|_{L^\infty(\R^n)} <\infty$.

\begin{rem}
We define $R(x,I)$ as
\begin{equation*}
\begin{cases}
R(x,2I_M)-I+2I_M & \text{when $I > 2I_M$},\\
R(x,I) & \text{when $-2I_M\leq I \leq 2I_M$},\\
R(x,-2I_M)-2I_M-I &\text{when $I < -2I_M$},
\end{cases}
\end{equation*}
so that $R(x,I)$ is continuously extended. Later on, we will see that $I(t)\in[0,I_M]$ for $t\geq0$. 
\end{rem}
Here is the organization of this paper. In Section 2, we construct a sequence of solution pairs $(u^\ep, I^\ep)$ for a \textit{relaxed equation} using Banach's fixed point argument. In Section 3, we show that $(u^\ep,I^\ep)$ converges to a solution pair $(u,I)$ for the original equation (\ref{original}) up to subsequences. In Section 4, uniqueness result for (\ref{original}) is presented when the reaction $R(x,I)$ is assumed to be separable in variables $x$ and $t$. In Section 5, we finish this paper by giving an example where uniqueness fails when $R(x,I)$ is not strictly decreasing in $I$.
  
\section{Construction of a solution of relaxed problem via a fixed point argument}
We first provide some regularity properties for 
   \begin{equation} \label{alter}
     \begin{cases}
        u_t =H(Du)+R(x,I(t)) &\text{in }\R^n \times [0,T],\\
  	 u(x,0)=u_0(x) &\text{on }\R^n,
    \end{cases}
  \end{equation}
where $I(t)$ is a given continuous function on $[0,T]$. It is known that there exists a unique viscosity to $u(x,t)$ for (\ref{alter}) which is bounded uniformly continuous in $\R^n \times [0,T]$.

\begin{thm}\label{space}
Let u be a unique viscosity solution of (\ref{alter}) for a given continuous function $I(t)$. Set $L=\sup_{t \in [0,T]} \|R(\cdot, I(t)\|_{\W(\R^n)}$. Then, for $T>0$, 
\begin{equation*}
\|Du\|_{L^\infty(\R^n \times [0,T])} \leq (L+1)T+\|Du_0\|_{L^\infty(\R^n)}.
\end{equation*}
\end{thm}

\begin{proof}
We follow arguments presented in \cite{Guide, Book, Uniqueness, lip}. We first define 
\begin{equation*}
C(t)=(L+1)t+K
\end{equation*}
where $K=\|u_0\|_{\W(\R^n)}$ and assume the solution $u$ is not Lipschitz continuous in space $x$. In other words, there exists $\sigma>0$ such that
\begin{equation*}
\sup_{x,y\in \R^n, t\in [0,T]} \left( u(x,t)-u(y,t)-C(t)|x-y| \right)=\sigma.
\end{equation*}

 We then define $\Phi$ as
\begin{equation*}
\Phi(x,y,t,s):=u(x,t)-u(y,s)-C(t)|x-y|-\frac{1}{{\alpha}^2}|t-s|^2-\beta(|x|^2+|y|^2).
\end{equation*}
for $(x,y,t,s) \in \R^{2n} \times [0,T]^2$ and $\alpha, \beta>0$. Since $u$ is bounded, we can find $(\ol x, \ol y, \ol t, \ol s) \in {\R}^{2n}\times [0,T]^2$ such that 
\begin{equation*}
\max_{(x,y,t,s) \in \R^{2n}\times [0,T]^2} \Phi(x,y,t,s) = \Phi(\ol x, \ol y, \ol t, \ol s).
\end{equation*}
We can also note that 
\begin{align*}
 \Phi(\ol x, \ol y, \ol t, \ol s) &\geq \max_{x,y \in \R^n , t \in [0,T]} \Phi(x,y,t,t) \\
&=\sup_{x,y\in\R^n, t\in[0,T]}u(x,t)-u(y,t)-C(t)|x-y|-2\beta|x|^2,
\end{align*}
which yields 
\begin{equation*}
\Phi(\ol x, \ol y, \ol t, \ol s) > \frac{\sigma}{2}
\end{equation*}
for $\beta$ small enough regardless of $\alpha$. Moreover, $\ol x \neq \ol y$ for $\alpha$ small enough. If not,
\begin{equation*}
\max_{(x,y,t,s) \in \R^{2n} \times [0,T]^2} \Phi(x,y,t,s) = \Phi(\ol x, \ol x, \ol t, \ol s)=u(\ol x,\ol t)-u(\ol x, \ol s)-\frac{1}{\alpha^2}|\ol t-\ol s|^2 -2\beta|\ol x|^2 \leq \frac{\sigma}{4}
\end{equation*}
for $\alpha$ small enough as $u\in BUC (\R^n \times [0,T])$.

We use $\Phi (\ol x, \ol y, \ol t, \ol s) \geq \Phi(0,0,0,0)$ to get
\begin{align*}
\frac{1}{\alpha^2} |\ol t -\ol s|^2 +\beta (|\ol x|^2+|\ol y|^2) &\leq u(\ol x, \ol t)-u(\ol y, \ol s)-u(\ol x, 0)+u(\ol y, 0)-C(\ol t)|\ol x-\ol y|\\ &\leq u(\ol x, \ol t)-u(\ol y, \ol s)-u(\ol x, 0)+u(\ol y, 0).
\end{align*}
The inequality above implies $|\ol t-\ol s| =O(\alpha)$, $|\ol x|, |\ol y|=O(1/\sqrt \beta)$ since $u$ is bounded. Moreover, $\ol t, \ol s$ have to be away from 0 since 
\begin{equation*}
\frac{\sigma}{2} < \Phi(\ol x, \ol y, \ol t, \ol s) < u(\ol x, \ol t)-u(\ol y, \ol s)-C(\ol t) |\ol x - \ol y|
\end{equation*}
and
\begin{equation*}
u(\ol x,0)-u(\ol y,0) \leq K|\ol x -\ol y|
\end{equation*}
where $K=C(0)$.

Observing that
$u(x,t)-\phi(x,t)$ has maximum at $(\ol x, \ol t)$ where
\begin{equation*}
\phi(x,t):=u(\ol y, \ol s)+C(t)|x - \ol y|+\frac{1}{{\alpha}^2}|t-\ol s|^2+\beta(|x|^2+|\ol y|^2).
\end{equation*}
By the definition of viscosity subsolutions, 
\begin{equation}\label{sub}
(L+1)|\ol x- \ol y|+\frac{2}{\alpha^2}(\ol t -\ol s) \leq H\left(C(\ol t) \frac{\ol x -\ol y}{|\ol x - \ol y|} + 2\beta \ol x \right) +R(\ol x, I(\ol t)).
\end{equation}
Similarly, $u(y,t)-\eta(y,t)$ has minimum at $(\ol y, \ol s)$ where
\begin{equation*}
\eta(y,s):=u(\ol x, \ol t)-C(\ol t)| \ol x - y|-\frac{1}{{\alpha}^2}|\ol t- s|^2- \beta(|\ol x|^2+|y|^2).
\end{equation*}
By the definition of viscosity supersolutions, 
\begin{equation} \label{sup}
\frac{2}{\alpha^2}(\ol t -\ol s) \geq H\left(C(\ol t) \frac{\ol x -\ol y}{|\ol x - \ol y|} - 2\beta \ol y \right) +R(\ol y, I(\ol s))
\end{equation}
Subtracting (\ref{sup}) from (\ref{sub}) gives us
\begin{align*}
(L+1)|\ol x -\ol y|
&\leq  H\left(C(\ol t) \frac{\ol x -\ol y}{|\ol x - \ol y|} + 2\beta \ol x \right) -H\left(C(\ol t) \frac{\ol x -\ol y}{|\ol x - \ol y|} - 2\beta \ol y \right) \\&+R(\ol x, I(\ol t))-R(\ol y, I(\ol s))\\
&\leq 2A\beta|\ol x+\ol y|+R(\ol x, I(\ol t))-R(\ol y, I(\ol s))
\end{align*}
where $A>0$ is a local Lipschitz constant for the Hamiltonian $H(p)$. Here, we can choose such $A$ since the terms inside of the Hamiltonian are not growing to either $\infty$ or $-\infty$. We note that
\begin{align*}
R(\ol x, I(\ol t))-R(\ol y, I(\ol s))&=R(\ol x, I(\ol t))-R(\ol x, I(\ol s))+R(\ol x, I(\ol s))-R(\ol y, I(\ol s)).\\
&\leq K_1|I(\ol t)-I(\ol s)| + \|R(\cdot, I(t))\|_{\W (\R^n)}|\ol x-\ol y|
\end{align*}
Then, taking $\alpha$ to 0 and combining previous two above gives
\begin{equation*}
(L+1)|\ol x -\ol y|\leq 2A\beta|\ol x+\ol y|+L|\ol x-\ol y|.
\end{equation*}
Finally, sending $\beta$ to 0 to give us
\begin{equation*}
L+1 \leq L,
\end{equation*}
which is a contradiction. Therefore, 
\begin{equation*}
\sup_{x,y\in \R^n, t\in [0,T]} u(x,t)-u(y,t)-C(t)|x-y| \leq 0.
\end{equation*}
By symmetry, we get, for $x,y\in \R^n$, $t\in[0,T]$,
\begin{equation*}
|u(x,t)-u(y,t)| \leq C(t)|x-y|,
\end{equation*}
and it is consistent with Lipschitz bound for $u_0$ as $C(0)=K$.
\end{proof}
\begin{thm}\label{time}
Assume $I$ is given and let u be a unique viscosity solution of (\ref{alter}). Then, for a positive C(T) depending on T,
\begin{equation*}
\|u_t\| _{L^\infty(\R^n \times [0,T])} <C(T).
\end{equation*}
\end{thm}

\begin{proof} We first show $-C\leq u_t \leq C$ for $(x,t) \in \R^n \times [0,T]$ in viscosity sense for a positive constant $C$ depending only on $T$ when $u$ is the viscosity solution of (\ref{alter}). Let us assume that $\phi(x,t)\in C^1(\R^n \times (0,\infty)$ touches $u(x,t)$ from above so that $u-\phi$ has maximum at ($x_0,t_0)\in \R^n \times (0,T]$. For $\phi$ to touch $u$ from above at $(x_0,t_0)$,
\begin{equation*}
\limsup_{(y,s)\rightarrow(x_0,t_0)} \frac{u(y,s)-u(x_0,t_0)-(D\phi, \phi_t)\cdot(y-x_0,s-t_0)}{|(y,s)-(x_0,t_0)|}\leq 0.
\end{equation*}
Let us take a sequence $(x',t_0)$ converging to $(x_0, t_0)$ such that
\begin{equation*}
-\frac{D\phi(x_0,t_0)}{|D\phi(x_0,t_0)|}= \frac {x'-x_0}{|x'-x_0|}
\end{equation*} 
Then we have
\begin{align*}
0 &\geq \limsup_{(y,s)\rightarrow(x_0,t_0)} \frac{u(y,s)-u(x_0,t_0)-(D\phi, \phi_t)\cdot(y-x_0,s-t_0)}{|(y,s)-(x_0,t_0)|}\\
&\geq \limsup_{(x',t_0)\rightarrow(x_0,t_0)} \frac {u(x',t_0)-u(x_0,t_0)-D\phi(x',t_0)\cdot (x'-x_0)}{|x'-x_0|}\\
&\geq -C(t_0)+|D\phi(x_0,t_0)|
\end{align*} 
where $C(t)=(L+1)t+K$ from Theorem \ref{space}.
Hence, $|D\phi|$ is bounded depending only on $T$ regardless of choice of test functions $\phi$.\\
By the definition of viscosity subsolutions, we obtain
\begin{equation*}
\phi_t \leq H(D\phi)+R(x_0,I(t_0)) \leq C
\end{equation*}
for a positive constant $C$ which depends on $T$. Therefore,  $u_t \leq C$ in viscosity sense. One can show the other inequality similarly. 

It remains to check that the inequality above in viscosity sense implies Lipschitz continuity of $u$ in time. Although elementary, we present the proof of it. Let us fix the time $s$ and define $u_1$ and $u_2$ as
\begin{align*}
u_1(x,t)&=u(x,t+s),\\
u_2(x,t)&=u(x,s)+Ct,
\end{align*}
so that $u_1(x,0)=u_2(x,0)$. Then, $u_1$ is a viscosity subsolution of $u_t=C$ while $u_2$ is a viscosity solution of $u_t=C$. Therefore, we have
\begin{equation*}
u_1(x,t) \leq u_2(x,t) \text{ for }(x,t)\in \R^n \times [0,T] 
\end{equation*}
by comparison principle, which implies
\begin{equation*}
u(x,t+s) \leq u(x,s) +Ct.
\end{equation*}
Similarly, we can derive
\begin{equation*}
u(x,s) - Ct \leq u(x,t+s),
\end{equation*}
which finishes the proof.
\end{proof}

\begin{defn}
For $\ep >0$ and $I \in C([0,T])$ such that $I(0)=0$, we define a mapping $\Sigma$ from $W=\{I \in C([0,T]), I(0)=0 \}$ to itself as following
\begin{equation*}
\Sigma:I(t) \mapsto (1-\ep)I(t)+\sup_ { \R^n} u^\ep(\cdot, t)
\end{equation*}
where $u^\ep(x,t)$ is a unique bounded uniformly continuous viscosity solution of (\ref{alter}) corresponding to $I(t)$. 
\end{defn}

The mapping $\Sigma$ is well defined as $\sup_{\R^n} u^\ep(\cdot,t)$ is continuous in time $t$ due to Lipschitz regularity properties of the viscosity solution $u^\ep(x,t)$. We first prove the following proposition before we show $\Sigma$ is contraction mapping.
\begin{prop}\label{prop1}
 Let $(u_1,I_1)$, $(u_2,I_2)$ be two solution pairs for (\ref{alter}). Then, for $f(t):=\|(u_1-u_2)(\cdot,t)\|_{L^\infty (\R^n)}$, $
f'(t) \leq K_1|I_1(t)-I_2(t)|
$
in viscosity sense where $K_1$ is a constant from {\rm (A1)}.
\end{prop}
\begin{proof}
Let $(u_1,I_1)$, $(u_2,I_2)$ be two viscosity solution pairs satisfying (2.1). We first prove that 
	\begin{equation*}
	\frac{d}{dt}\|(u_1-u_2)(\cdot,t)\|_{L^\infty (\R^n)} \leq K_1 |I_1(t)-I_2(t)| 
	\end{equation*}
in viscosity sense. Clearly $f(t)$ is Lipschitz continuous as both $u_1$, $u_2$ are bounded and Lipschitz continuous. Without loss of generality, we may assume there exists $\phi(t)\in C^1((0,T])$ for which $f(t)-\phi(t)$ has a  strict maximum at $t_0>0$ as 0 so that $\phi(t_0)=f(t_0)$ and $f(t_0)=\sup_{ \R^n} (u_1-u_2)(\cdot,t_0)\geq0$. We can also assume that $\phi \geq0$.

Now for $\lam>0$ and $\ep>0$ we define
	\begin{align*}
	\Phi(x,y,t,s):=u_1(x,t)-u_2(y,s)+\lam(t+s)&-\frac{1}{2}(\phi(t)+\phi(s))\\
	&-\frac{1}{\ep^2}(|t-s|^2+|x-y|^2)-\ep(|x|^2+|y|^2).
	\end{align*}
Since $u_1,u_2$ are bounded, for $\lam>0$ given, there exists $(x_\ep,y_\ep,t_\ep,s_\ep)\in \R^{2n} \times [0,T]$ such that 
	\begin{equation}\label{lem}
	\max_{(x,y,t,s)\in \R^{2n}\times [0,T]^2} \Phi(x,y,t,s)=\Phi(x_\ep,y_\ep,t_\ep,s_\ep)>\lam t_0
	\end{equation}
for all $\ep$ small enough. To verify this, for $\lam>0$ given, we choose $\del>0$ to be an arbitrarily small positive number such that $2\lam t_0 - \del>0$. There also exists $x'\in \R^n$ such that $u_1(x',t_0)-u_2(x',t_0)-\phi(t_0)>-\del$. By observing 
	\begin{equation*}
	\Phi(x_\ep,y_\ep,t_\ep,s_\ep)\geq \Phi(x',x',t_0,t_0)>2\lam t_0 -\del -2\ep|x'|^2,
	\end{equation*}
we deduce 
	\begin{equation*}
	\liminf_{\ep\rightarrow 0} \Phi(x_\ep,y_\ep,t_\ep,s_\ep)\geq 2\lam t_0
	\end{equation*}
since $\del$ is arbitrary. Hence, we can conclude by
\begin{equation*}
	 \Phi(x_\ep,y_\ep,t_\ep,s_\ep)>\lam t_0
\end{equation*}
for all $\ep$ small enough, which verifies (\ref{lem}).

From $\Phi(x_\ep,y_\ep,t_\ep,s_\ep)\geq\Phi(0,0,0,0)$,
we get
\begin{align*}
u_1(x_\ep, t_\ep)-u_2(y_\ep,s_\ep)+\lam(t_\ep+s_\ep)&-\frac{1}{2}(\phi(t_\ep)+\phi(s_\ep))-\frac{1}{\ep^2}(|t_\ep-s_\ep|^2+|x_\ep-y_\ep|^2)\\
&-\ep(|x_\ep|^2+|y_\ep|^2) \geq u_1(0,0)-u_2(0,0)-\phi(0),
\end{align*}
which implies
	\begin{equation*}
	\frac{1}{\ep^2}(|t_\ep-s_\ep|^2+|x_\ep-y_\ep|^2)+\ep(|x_\ep|^2+|y_\ep|^2)<\infty
	\end{equation*}
since $u_1,u_2,\phi$ are all bounded. Hence, we obtain
\begin{equation*}
|x_\ep-y_\ep|, |t_\ep-s_\ep|=O(\ep), \text{ }|x_\ep|,|y_\ep|=O(1/\sqrt \ep).
\end{equation*}
Similarly, we use $\Phi(x_\ep,y_\ep,t_\ep,s_\ep)\geq \Phi(x_\ep,x_\ep,t_\ep,t_\ep)$ to have

\begin{multline*}
\frac{1}{\ep^2}(|t_\ep-s_\ep|^2+|x_\ep-y_\ep|^2) \leq u_2(x_\ep,t_\ep)-u_2(y_\ep,s_\ep)+\lam(s_\ep-t_\ep)\\+\frac{1}{2}(\phi(t_\ep)-\phi(s_\ep))+\ep(x_\ep-y_\ep)\cdot(x_\ep+y_\ep).
\end{multline*}

It follows that $|t_\ep-s_\ep|$, $|x_\ep-y_\ep|$=$o(\ep)$ since $u_2$ is uniformly continuous. Additionally, from $\Phi(x_\ep,y_\ep,t_\ep,s_\ep)\geq\Phi(x_\ep,x_\ep,t_\ep,s_\ep)$, we obtain
\begin{equation*}
u_2(x_\ep,s_\ep)-u_2(y_\ep,s_\ep)+\ep(x_\ep-y_\ep)\cdot(x_\ep+y_\ep)\geq \frac{1}{\ep^2}|x_\ep-y_\ep|^2
\end{equation*}
and by Lipschitz continuity of $u_2$ in space, boundedness of $\frac{1}{\ep^2}|x_\ep-y_\ep|$ follows given that $\ol x \neq \ol y$. Even if $\ol x=\ol y$, we can claim the same bound. Moreover, from (\ref{lem}), we have 
\begin{equation*}
0<\lam t_0< u_1(x_\ep, t_\ep)-u_2(y_\ep,s_\ep)+\lam(t_\ep+s_\ep)
\end{equation*}
for all $\ep$ small enough, we deduce that there exists $\mu>0$ such that $t_\ep$, $s_\ep \geq \mu$.\\

Noticing that $u_1(x,t)-\eta_1(x,t)$ achieves maximum at $(x_\ep,t_\ep)$ where
\begin{equation*}
\eta_1(x,t):=u_2(y_\ep,s_\ep)-\lam(t+s_\ep)+\frac{1}{2}(\phi(t)+\phi(s_\ep))+\frac{1}{\ep^2}(|t-s_\ep|^2+|x-y_\ep|^2)+\ep(|x|^2+|y_\ep|^2),
\end{equation*}
we can use the viscosity subsolution test to get
\begin{equation}\label{subs}
-\lam+\frac{2}{\ep^2}(t_\ep-s_\ep)+\frac{1}{2}\phi'(t_\ep)\leq H\left(\frac{2}{\ep^2}(x_\ep-y_\ep)+2\ep x_\ep\right)+R(x_\ep,I_1(t_\ep)).
\end{equation}
Similarly, $u_2(y,s)-\eta_2(y,s)$ achieves minimum at $(y_\ep,s_\ep$) where 
\begin{equation*}
\eta_2(y,s)=u(x_\ep,t_\ep)+\lam(t_\ep+s)-\frac{1}{2}(\phi(t_\ep)+\phi(s))-\frac{1}{\ep^2}(|t_\ep-s|^2+|x_\ep-y|^2)-\ep(|x_\ep|^2+|y|^2).
\end{equation*}
Again by the viscosity supersolution test, we get
\begin{equation} \label{super}
\lam+\frac{2}{\ep^2}(t_\ep-s_\ep)-\frac{1}{2}\phi'(s_\ep)\geq H\left(\frac{2}{\ep^2}(x_\ep-y_\ep)-2\ep y_\ep\right)+R(y_\ep,I_2(s_\ep)).
\end{equation}
Subtracting (\ref{super}) from (\ref{subs}) results in
\begin{align*}
-2\lam +\frac{1}{2}(\phi'(t_\ep)+\phi'(s_\ep)) &\leq H\left(\frac{2}{\ep^2}(x_\ep-y_\ep)+2\ep x_\ep\right)-H\left(\frac{2}{\ep^2}(x_\ep-y_\ep)-2\ep y_\ep\right)\\
&+R(x_\ep,I_1(t_\ep))-R(y_\ep,I_2(s_\ep)).
\end{align*}
Using the fact that $\frac{2}{\ep^2}(x_\ep-y_\ep)$ is bounded. Because of Theorem \ref{space} and $H(p)$ locally Lipschitz continuous, we have
\begin{equation*}
\phi'(t_0)\leq K_1|I_1(t_0)-I_2(t_0)|
\end{equation*}
if $t_\ep$ and $s_\ep$ converge to $t_0$ as $\ep$ vanishes as $\lam>0$ is arbitrary. Hence, it is enough to prove that $t_\ep$, $s_\ep$ actually converge to $t_0$ as $\ep$ goes to 0 up to subsequence for a given $\lam>0$. Let us suppose that $t_\ep$, $s_\ep$ converge to $t'_0\neq t_0$ up to subsequence with respect to $\ep$. Then there exists $\gam>0$ such that $\sup_{x\in \R^n} (u_1-u_2)(\cdot,t'_0)-\phi(t'_0)\leq -\gam$ since $f(t)-\phi(t)$ obtains a strict maximum 0 at $t_0$.

We now observe that
\begin{align*}
\Phi(x_\ep,y_\ep,t_\ep,s_\ep) &\leq u_1(x_\ep,t_\ep)-u_2(y_\ep,s_\ep)+\lam(t_\ep+s_\ep)-\frac{1}{2}(\phi(t_\ep)+\phi(s_\ep))\\
&=(u_1(x_\ep,t_\ep)-u_1(x_\ep,t'_0))-(u_2(y_\ep,s_\ep)-u_2(y_\ep,t'_0))+(\lam(t_\ep+s_\ep)-2\lam t'_0)\\&-\frac{1}{2}(\phi(t_\ep)+\phi(s_\ep))+\phi(t'_0)+u_1(x_\ep,t'_0)-u_2(y_\ep,t'_0)+2\lam t'_0-\phi(t'_0).
\end{align*}
We then take $\liminf_{\ep\rightarrow 0}$ on both sides, the following is obtained;
\begin{equation*}
\lam t_0\leq -\gam+2\lam t'_0.
\end{equation*}
However, taking $\lam$ to 0 yields that $0 \leq -\gam$, which is a contradiction. Therefore, $t_\ep$, $s_\ep$ must converge to $t_0$. 
\end{proof}

\begin{prop}
The map $\Sigma:W \mapsto W$ is a contraction mapping for a short time $T'=\frac{2\ep}{K_1}>0$, and there exists a viscosity solution of
\begin{equation*}
    \begin{cases}
       u_t^\ep=H(D u^\ep)+R(x,I^\ep(t)) &\text{in } \R^n \times [0,T],\\
       \ep I^\ep(t)=\sup_{\R^n} u^\ep( \cdot ,t) &\text{on } [0,T],\\
	 I^\ep(0)=0, \\	 
	 u^\ep(x,0)=u_0(x) &\text{on } \R^n
    \end{cases}
\end{equation*}
for $I^\ep \in W$. We call this a \textit{relaxed equation}.
\end{prop}
\begin{proof}
First of all, the proposition (\ref{prop1}) results in
\begin{equation}\label{rel}
\|(u_1-u_2)(\cdot,t)\|_{\W(\R^n)} \leq K_1 \|I_1-I_2\|_{L^\infty([0,T])}.
\end{equation}
Now, let $I_1$ and $I_2$ be in $W$ and $T'=\frac{2\ep}{K_1}$. Observing
\begin{align*}
\Sigma(I_2)-\Sigma(I_1) &\leq (1-\ep)(I_2(t)-I_1(t))+\sup_{\R^n} u_2(\cdot,t)-\sup_{\R^n} u_2(\cdot,t)\\
&\leq (1-\ep)(I_2(t)-I_1(t))+|\sup_{\R^n} u_2(\cdot,t)-\sup_{\R^n} u_2(\cdot,t)|,
\end{align*}
and combining with (\ref{rel}), we have 
\begin{align*}
\|\Sigma(I_2)-\Sigma(I_1)\|_{L^\infty([0,T'])} &\leq (1-\ep+K_1T')\|I_2-I_2\|_{L^\infty([0,T'])}\\
&=(1-\frac{\ep}{2})\|I_2-I_2\|_{L^\infty([0,T'])},
\end{align*}
which implies $\Sigma$ is a contraction mapping. 

For a fixed $\ep>0$ and time $T'$,  by Banach's fixed point theorem, there exists a unique fixed point $I^\ep \in W$ such that
\begin{align*}
&I^\ep=\Sigma(I^\ep)=(1-\ep)I^\ep+\sup_{\R^n} u^\ep(\cdot,t)\\
&\Leftrightarrow \ep I^\ep=\sup_{\R^n} u^\ep(\cdot,t)
\end{align*}

Therefore, for the short time $T'=\frac{2\ep}{K_1}$, we have a solution pair $(u^\ep,I^\ep)$ for
\begin{equation*}
	\begin{cases}
       u_t^\ep=H(Du^\ep)+R(x,I^\ep(t)) &\text{in } \R^n \times [0,T'],\\
       \ep I^\ep(t)=\sup_{\R^n} u^\ep(\cdot ,t) &\text{on } [0,T'],\\
	 I^\ep(0)=0, \\	 
	 u^\ep(x,0)=u_0(x) &\text{on } \R^n.
    \end{cases}
\end{equation*}
We also notice that that $T'$ depends only on $K_1$ and $\ep$, applying the argument above successively on time intervals $[0,T']$, $[T',2T']$, $[2T',3T'], \cdot \cdot \cdot$, one can obtain a solution which is valid for whole time interval $[0,T]$.
\end{proof}  

\section{Limiting equation}
In the previous section, for each $\ep>0$, we have constructed a solution pair $(u^\ep,I^\ep)$ to
\begin{equation}\label{approximation}
	\begin{cases}
       u_t^\ep=H(D u^\ep)+R(x,I^\ep(t)) &\text{in } \R^n \times [0,T],\\
       \ep I^\ep(t)=\sup_{\R^n} u^\ep(\cdot,t) &\text{on } [0,T],\\
	 I^\ep(0)=0, \\	 
	 u^\ep(x,0)=u_0(x) &\text{on } \R^n.
    \end{cases}
\end{equation}
By Theorems \ref{space} and \ref{time}, $u^\ep$ is Lipschitz continuous in both time and space but Lipschitz constants are not uniform in $\ep$. The constants rather depend on the bound of $I(t)$ as we can see in the proofs. In this section, we first prove $I^\ep$ is nondecreasing and uniformly bounded by $I_M$ regardless of $\ep$. Then it follows that Lipschitz constants in time and space for $u^\ep$ are uniform so that $u^\ep$ converges locally uniformly to a bounded Lipschitz continuous function $u$ up to subsequence of $\ep$ by the Arzela-Ascoli theorem. We finish this section by noting that a limit function $u$ actually solves the original constrained problem with $I(t)$ using the stability result for discontinuous Hamilton-Jacobi equations.

\begin{prop}
Let $(u^\ep,I^\ep)$ be a solution of (\ref{approximation}). Then $I^\ep(t)$ is nondecreasing in $t$.
\end{prop}
\begin{proof}
We first claim that $I^\ep(t)$ cannot have an interior strict local maximum. Let us suppose $I^\ep(t)$ obtains a strict local maximum at $t_0\in(0,T)$ so it satisfies 
\begin{equation}\label{supre}
\ep I^\ep(t_0)=\sup_{ \R^n} u^\ep(\cdot,t_0).
\end{equation}
For $\beta>0$, let us define 
\begin{equation*}
f(x,t):=u^\ep(x,t)-\ep I^\ep(t_0)-\beta \sqrt {1+|x|^2}
\end{equation*}
so that we can find $(x_\beta,t_\beta) \in \R^n \times [0,T]$ such that
\begin{equation*}
\max_{(x,t)\in \R^n \times [0,T]} f(x,t)=f(x_\beta,t_\beta)
\end{equation*}
Clearly, $t_\beta=t_0$ as $f(x,t)\leq f(x,t_0)$. For any positive $\delta>0$ given, from (\ref{supre}), one can find $y\in \R^n$, such that
\begin{equation*}
\ep I^\ep(t_0)-\delta <u^\ep(y,t_0)<\ep I^\ep(t_0),
\end{equation*}
which yields,
\begin{equation*}
f(x_\beta, t_0) \geq f(y,t_0)=u(y,t_0)-\ep I^\ep(t_0) -\beta \sqrt{1+|y|^2} \geq -\delta-\beta \sqrt{1+|y|^2}.
\end{equation*}

Now we take $\liminf$ on both sides, we get $\liminf_{\beta \rightarrow 0} f(x_\beta, t_0) \geq -\delta$ for any $\delta >0$. Moreover, it is clear that $f(x_\beta, t_0)\leq 0$. Combining these two, we get 
\begin{equation*}
\lim_{\beta \rightarrow 0} f(x_\beta,t_0) =0.
\end{equation*}
Additionally, one can derive
\begin{equation*}
\lim_{\beta \rightarrow 0}\beta \sqrt {1+|x_\beta|^2}=0
\end{equation*}
since
\begin{equation*}
0 \geq -\beta \sqrt {1+|x_\beta|^2} \geq f(x_\beta,t_0)
\end{equation*}
and $f(x_\beta,t_0) \rightarrow 0$ as $\beta$ goes to 0.
Therefore, we can conclude with
\begin{equation*}
\lim_{\beta \rightarrow 0} u^\ep(x_\beta,t_0) =\ep I^\ep(t_0).
\end{equation*}

We can also observe that $u(x,t)-\phi(x,t)$ obtains a local maximum at $(x_\beta, t_0) \in \R^n \times (0,T)$ where $\phi(x,t):=\ep I^\ep(t_0)-\beta \sqrt {1+|x|^2}$. Hence, we have  
\begin{equation*}
0 \leq H\left(\beta \frac{2x_\beta}{\sqrt {1+|x_\beta|^2}}\right)+R(x_\beta,I^\ep(t_0))
\end{equation*}
by the definition of viscosity subsolutions.
Taking $\liminf$ with respect to $\beta$, we have
\begin{equation}
\liminf_{\beta \rightarrow 0} R(x_\beta,I^\ep(t_0)) \geq 0
\end{equation}
since $\frac{2x_\beta}{\sqrt {1+|x_\beta|^2}}$ is bounded.
From (\ref{approximation}), we get the following inequalities  for $h>0$ small enough
\begin{align*}
\ep I^\ep(t_0+h)-u^\ep(x_\beta,t_0) &\geq u^\ep(x_\beta,t_0+h)-u^\ep(x_\beta,t_0)\\&=\int_{t_0}^{t_0+h}u^\ep_t(x_\beta,t)dt\\ &\geq \int_{t_0}^{t_0+h} R(x_\beta,I^\ep(t))dt\\
                                                     &\geq \int_{t_0}^{t_0+h} R(x_\beta, I^\ep(t))dt=hR(x_\beta,I^\ep(t_0)).
\end{align*}
Consequently, taking $\liminf$ with respect to $\beta$ yields
\begin{equation*}
\ep(I^\ep(t_0+h)-I^\ep(t_0))\geq \liminf_{\beta \rightarrow 0} hR(x_\beta, I^\ep(t_0)) \geq 0,  
\end{equation*}
which contradicts the fact that $I^\ep(t)$ achieves a strict local maximum at $t_0$. Hence, $I^\ep(t)$ cannot have an interior strict local maximum.

It remains to prove $I^\ep(t)$ is nonnegative and is even nondecreasing on $[0,T]$. Let us first assume that there exists $t_0\in (0,T)$ such that $I^\ep(t_0) <0$. If $I^\ep(t)$ is negative for all $t \in (0,t_0)$, we get a contradiction using the similar argument above with $t_0$ replaced by 0 since $R(x,I^\ep(t))>0$ when $I^\ep(t)<0$ and $I^\ep(0)=0$. Else if there exists $t_1 \in (0,t_0)$ such that $I^\ep(t_1)>0$, then $I^\ep(t)$ has an interior local maximum, which cannot happen. Therefore, $I^\ep(t)$ is nonnegative. 
Now it is easy to see $I^\ep$ is nondecreasing on $[0,T]$ by the following argument; if we can find $0<t_0< t_1$ in $(0,T)$ such that $I^\ep(t_0)>I^\ep(t_1)>0$, then $I^\ep(t)$ achieves interior local maximum as well because $I^\ep(0)=0$. Hence, we finish the proof.
\end{proof}
\begin{prop}
Let $(u^\ep,I^\ep)$ be a solution of (\ref{approximation}). Then $0\leq I^\ep(t) \leq I_M$ for $t\geq0$.
\end{prop}
\begin{proof}
We may assume that there exists $t_0 \in (0,T)$ at which $I^\ep(t)$ is differentiable and $I^\ep(t_0)>I_M$. It follows that 
\begin{equation}\label{bound}
R(x,I^\ep(t_0))<0.
\end{equation}
since $\max_{x \in \R^n} R(x, I_M)=0$.
We can also find $\phi(t) \in C^1(\R^+)$ such that $\ep I^\ep(t)-\phi(t)$ has a local maximum at $t_0$ and $\phi'(t_0) > 0$. For $\beta>0$. 

We now consider
\begin{equation*}
\Phi(x,t)=u^\ep(x,t)-\phi(t)-\beta \sqrt{1+|x|^2},
\end{equation*}
which has a maximum at $(x_\beta, t_0)$. By the definition of viscosity subsolutions, we have
\begin{equation*}
\phi'(t_0)\leq H\left(\beta \frac{2x_\beta}{\sqrt {1+|x_\beta|^2}}\right)+R(x_\beta, I^\ep(t_0)).
\end{equation*}
Therefore, we have
\begin{equation*}
0<\phi'(t_0)\leq\liminf_{\beta \rightarrow 0} R(x_\beta,I^\ep(t_0)) .
\end{equation*}
However, this contradicts (\ref{bound}). Therefore, $0 \leq I^\ep(t) \leq I_M$ since $I^\ep$ is nondecreasing and $I^\ep(0)=0$.   
\end{proof}

For a family of locally uniformly bounded functions $\{u_\alpha\}_{\alpha\in\R}$, we define upper(lower) half-relaxed limit $\ol u$(or $\ul u$) as
 \begin{equation*}
 \ol u ={\limsup_{\alpha\rightarrow\infty}} {^{\star}u}_\alpha (x):=\lim_{\alpha\rightarrow\infty} \sup\{u_\del (y):|x-y|\leq1/\beta \text{ where } \del,\beta\geq\alpha\}
 \end{equation*}
 and
 \begin{equation*}
 \ul u ={\liminf_{\alpha\rightarrow\infty}} {_{\star}u}_\alpha (x):=\lim_{\alpha\rightarrow\infty} \inf\{u_\del (y):|x-y|\leq1/\beta \text{ where } \del,\beta\geq\alpha\}
\end{equation*}

\begin{lem}
Upper half-relaxed limit $\ol I(t)$ and lower-half relaxed limit $\ul I(t)$ of $I^\ep(t)$ agree almost everywhere for $\{I^\ep\} \subset C([0,T])$ nondecreasing.
\end{lem}
 \begin{proof}
By Helly's compactness theorem, we may assume that 
\begin{equation*}
I^\ep(t) \rightarrow I(t) \text{ everywhere up to passing to a subsequence}
\end{equation*}
where $I(t)$ is nondecreasing on $[0,T]$. Moreover, $I(t)$ has only countably many jump discontinuities. We claim that 
\begin{equation*}
\ol I(t) \leq I(t+)
\end{equation*}
and
\begin{equation*}
\ul I(t) \geq I(t-).
\end{equation*}
For simplicity, we may assume $t=0$ and it is enough to prove the first case as proving the second case is pretty much similar. 

By the definition of upper half-relaxed limit and the property that $I^\ep$ is nondecreasing, we have
\begin{equation*}
\ol I(0) = \lim_{\gam \rightarrow 0} \sup_\beta \{I^\beta (\gam): \beta \leq \gam \}.
\end{equation*}
We note that $\sup_{\beta} \{I^\beta (\gam): \beta \leq \gam \}$ is decreasing in $\gam$. For $\del>0$ and a fixed $\gam_0>0$, we can choose 
$\beta_0 \leq \gam_0$ such that $|I^{\beta}(\gam_0)-I(\gam_0)| <\del$ for all $\beta \leq \beta_0$. Therefore, for $\gam>0$ given, we have
\begin{align*}
\ol I(0) &\leq \sup_\beta \{I^\beta (\beta_0): \beta \leq \beta_0 \}\\
&\leq \sup_\beta \{I^\beta (\gam_0): \beta \leq \beta_0 \}\\
&=\sup_\beta \{I^\beta (\gam_0)-I(\gam_0): \beta \leq \beta_0 \} +I({\gam_0})\\
&\leq \del +I(\gam_0).
\end{align*}
Since, $\del$ is arbitrary, taking $\gam_0$ to 0 yields 
\begin{equation*}
\ol I(0) \leq I(0+).
\end{equation*} 
Similarly, one can prove the other inequality, hence, we can conclude that
\begin{equation*}
\ol I(t) = \ul I(t) \text{ a.e.}
\end{equation*}
\end{proof}

\begin{thm}
There exists a pair $(u,I)$ solving (\ref{original}).
\end{thm} 
\begin{proof}
By the stability result for discontinuous Hamiltonian in \cite{Guide, Book}, $\ol u$ is a subsolution of
\begin{equation*}
u_t=H(D u)+R(x,\ul I(t))
\end{equation*}
and $\ul u$ is a supersolution of
\begin{equation*}
u_t=H(Du)+R(x,\ol I(t)).
\end{equation*}
Since $0\leq I^\ep \leq I_M$, there exists a subsequence $\{u^{\ep_j}\}_{j \in \N}$ such that
\begin{equation*}
u^{\ep_j} \rightarrow u \text{ locally uniformly on }\R^n \times [0,T]
\end{equation*}
Therefore, we can let $u=\ol u = \ul u$. Moreover,  $\ol I(t)=\ul I(t)$ almost everywhere, we let
\begin{equation*}
I(t)=\ol I(t) =\ul I(t).
\end{equation*}
With this new $I(t)$, $u$ is a viscosity solution of
   \begin{equation*}
     \begin{cases}
       u_t=H(D u)+R(x,I(t)) &\text{in } \R^n \times [0,T],\\
  	 u(x,0)=u_0(x) &\text{on } \R^n.
    \end{cases}
  \end{equation*}
Moreover, we obtain
\begin{equation*}
\sup_{ \R^n } u(\cdot, t)=0
\end{equation*}
from the relation
\begin{equation*}
\ep I^\ep(t)=\sup_{\R^n} u^\ep(\cdot,t)
\end{equation*}
together with locally uniform convergence of $u^\ep$. 
\end{proof}
\section{Uniqueness for a certain birth rate}
In this section, we deal with uniqueness of a pair ($u$, $I$) where $u(x,t)$ is a bounded uniformly continuous viscosity solution of equation (\ref{original}). We provide uniqueness result when $R(x,I)$ has certain structures. Here, we follow structural conditions in \cite{Dirac.C}. We have not been able to obtain full unconditional uniqueness result when $R(x,I)$ is not decoupled.
\begin{thm}
For the equation (\ref{original}), there exists a unique viscosity solution $u(x,t)$ and nonnegative increasing function $I(t)$ when the birth $R(x,I)$ can be written as either
\begin{equation*}
R(x,I(t))=b(x)-d(x)Q(I), \text{with $Q(I)>0$ increasing,}
\end{equation*}
or
\begin{equation*}
R(x,I(t))=b(x)Q(I)-d(x), \text{with $Q(I)>0$ decreasing,}
\end{equation*}
where $b(x), d(x) \in \W(\R^n)$ and there exists $b_m>0$ such that $b(x)>b_m$.
\begin{proof}
We only need to deal with the first case as handling the second case is similar. Again, we follow the argument by B. Perthame and G. Barles in \cite{Dirac.C}. Let us assume that there are two viscosity solutions $u_1$ and $u_2$ corresponding to $I_1(t)$ and $I_2(t)$ respectively. In other words, $u_i's$ satisfy
   \begin{equation*}
     \begin{cases}
       (u_i)_t=H(D u_i)+R(x,I_i(t)) & \text{in } \R^n \times [0,T],\\
       \sup_{\R^n} u_i( \cdot ,t)=0 &\text{on } [0,T],\\
	 I(0)=0,\\
  	 u_i(x,0)=u_0(x) & \text{on }\R^n,
    \end{cases}
  \end{equation*}
for $i=1,2$ in viscosity sense. Now we consider
\begin{align*}
\Psi_i = u_i(x,t)-b(x)\Sigma_i(t)\\
\Sigma_i(t)=\int_{0}^{t} Q(I_i(s))ds
\end{align*}
and they satisfy
\begin{equation*}
(\Psi_i)_t = -d(x) +H\left(D(\Psi_i +b(x)\Sigma_i(t))\right)
\end{equation*}
for $i=1,2$ in viscosity sense. Using similar argument in Proposition \ref{prop1}, we have
\begin{equation}\label{ineq}
\frac{d}{dt}\|\Psi_1-\Psi_2(\cdot,t)\|_{L^\infty(\R^n)} \leq C|\Sigma_1(t)-\Sigma_2(t)|
\end{equation}
in viscosity sense  for a positive $C$. 

Noting that $\sup_{x\in\R^n} u_1(x,t)=0$, for any $\del>0$ given, we can find $y\in \R^n$ such that 
\begin{equation*}
-\del \leq u_1(y,t) \leq 0.
\end{equation*}
By the following inequalities,
\begin{align*}
-\del &\leq u_1(y,t) - \sup_{ \R^n}u_2(\cdot,t)\\
&\leq u_1(y,t)-u_2(y,t)\\
&=b(y)[\Sigma_1(t)-\Sigma_2(t)]+\Psi_1(y,t)-\Psi_2(y,t)\\
&\leq b(y)[\Sigma_1(t)-\Sigma_2(t)]+\sup_{ \R^n} [\Psi_1(\cdot,t)-\Psi_2(\cdot,t)],
\end{align*}
we get
\begin{equation*}
b(y)[\Sigma_2(t)-\Sigma_1(t)] \leq \sup_{\R^n} [\Psi_1(\cdot,t)-\Psi_2(\cdot,t)]  +\del.
\end{equation*}
Here we may assume $\Sigma_2(t)-\Sigma_1(t)$ is positive. Then
\begin{align*}
b_m|\Sigma_2(t)-\Sigma_1(t)|&=b_m[\Sigma_2(t)-\Sigma_1(t)]\\
&\leq b(y)[\Sigma_2(t)-\Sigma_1(t)]\\
&\leq \sup_{x\in \R^n} |\Psi_1(\cdot,t)-\Psi_2(\cdot,t)|  +\del.
\end{align*}
Since $\del>0$ is arbitrary, we have
\begin{equation*}
b_m|\Sigma_2(t)-\Sigma_1(t)| \leq \sup_{ \R^n} |\Psi_1(\cdot,t)-\Psi_2(\cdot,t)|
\end{equation*}
by switching the role of $\Sigma_1$ and $\Sigma_2$ if necessary.
Combining with (\ref{ineq}) yields
\begin{equation*}
\frac{d}{dt}\|(\Psi_1-\Psi_2)(\cdot,t)\|_{L^\infty(\R^n)} \leq C\|(\Psi_1-\Psi_2)(\cdot,t)\|_{L^\infty(\R^n)} .
\end{equation*}
Consequently, uniqueness follows by Gronwall's inequality.
\end{proof}
\end{thm}

\section{Nonuniqueness result}
If $R(x,I)$ is not strictly decreasing in $I$, we can give an example of nonuniqueness. 
\begin{thm}
Let $I(t)$ be a nondecreasing continuous function such that $I(0)=0$ and $I(t)\geq0$. Assume $R(x,I)$ is defined as
 \begin{equation*}
     \begin{cases}
       R(x,I)=0 & |x| \geq \frac{1}{2},\\
       R(x,I)=(1-|x|)(1-I(t)) & |x|\leq \frac{1}{2},
    \end{cases}
  \end{equation*}
and $u_0(x)$ satisfies
\begin{equation*}
    \begin{cases}
      0 & |x| \geq 2,\\
       -(\frac{x-2}{2})^2  & 1\leq x \leq 2,\\
       -(\frac{x+2}{2})^2 & -2\leq x \leq -1,\\
       \frac{1}{2}x^2-\frac{3}{4} & |x|\leq1.
    \end{cases}
\end{equation*}
Then, for $T>0$ small enough, there exist infinitely many viscosity solutions to
   \begin{equation}\label{comp}
     \begin{cases}
       u_t=|Du|^2+R(x,I(t)) &\text{in }\R^n \times [0,T],\\
  	 u(x,0)=u_0(x) &\text{on }\R^n,      
    \end{cases}
  \end{equation}
satisfying 
\begin{equation*}\label{const}
\sup_{\R^n} u(\cdot,t)=0 \text{ on }[0,T] 
\end{equation*}
\end{thm}
\begin{proof}
Let $I(t)$ be a continuous nondecreasing function such that $I(0)=0$ and $c>0$ satisfy $c \geq 1-I(t)$. Also we denote a unique viscosity solution of (\ref{comp}) by $u(x,t)$.  Now let us define $v_c(x,t)$ as
\begin{equation*}
 v_c(x,t)=\begin{cases}
 0 & |x| \geq 2,\\
 -(\frac{x-2}{2})^2  & 1+ct \leq x \leq 2,\\
  -\frac{1}{4}+ct & |x| \leq 1+ct,\\
  -(\frac{x+2}{2})^2 & -2\leq x \leq -1-ct,
 \end{cases}
\end{equation*}
and $w_c(x,t)$ as
\begin{equation*}
 w_c(x,t)=\begin{cases}
 0 & |x| \geq 2,\\
 -(\frac{x-2}{2})^2  & 2-\frac{\sqrt 3}{2}-ct \leq x \leq 2,\\
  -\frac{3}{4}-ct & |x| \leq 2-\sqrt 3-ct,\\
  -(\frac{x+2}{2})^2 & -2\leq x \leq -(2-\sqrt 3)+ct.
\end{cases}
\end{equation*}
It is straightforward to check $w_c$ is a viscosity subsolution to (\ref{comp}) since there is no test function touching from above at $|x|=1+ct$. On the other hand, as (\ref{comp}) is a concave Hamilton-Jacobi equation and a Lipschitz continuous function $v_c$ solves (\ref{comp}) almost everywhere sence. By Proposition 7.25 in \cite{Book}, $v_c$ is a viscosity supersolution. Therefore, we have
\begin{equation*}
w_c(x,0)\leq u(x,0)\leq v_c(x,0).
\end{equation*}
It follows that
\begin{equation*}
w_c(x,t)\leq u(x,t)\leq v_c(x,t)
\end{equation*}
by comparison principle. Moreover, for $T>0$ small,
\begin{equation*}
\sup_{\R^n} v_c(\cdot,t)=\sup_{ \R^n} w_c(\cdot,t)=0 \text{ for } 0\leq t \leq T.
\end{equation*}
Therefore, $u$ satisfies the constraint condition (\ref{const}). As we can repeat that process for any choice of $I(t)$ and $c$, infinitely many solution pairs $(u,I)$ are generated.
\end{proof}

\section{Conclusion}
We presented a new way of building a viscosity solution of Hamilton-Jacobi equation with an unknown function $I(t)$ and a supremum constraint via a fixed point argument. We also provided that when $R(x,I)$ is separable in $x$ and $t$, the solution is unique for any nonnegative, locally Lipschitz continuous Hamiltonian $H(p)$ satisfying $H(0)=0$. Here, we do not need convexity of the Hamiltonian. On the other hand, many solutions can be generated when the reaction $R(x,I)$ fails to strictly decreasing with respect to resource, $I$. Up to now, uniqueness of a solution pair $(u,I)$ corresponding to a general reaction $R(x,I)$ is still open. In the recent work by Mirrahimi and Roquejoffre \cite{Uniqueness}, it was proved that the solution is unique under restrictive assumptions that Hamiltonian and initial condition are uniformly concave and satisfy some further structural assumptions using optimal control formulation. We plan to investigate this matter in the near future.

\end{document}